\documentclass[11pt,leqno]{amsart}
\usepackage{amsmath}
\usepackage{amsfonts}
\usepackage{amssymb}
\usepackage{graphicx}
\usepackage{color}
\usepackage[all,cmtip]{xy}

\newtheorem{theorem}{Theorem}[section]

\newtheorem{corollary}[theorem]{Corollary}

\newtheorem{definition}[theorem]{Definition}

\newtheorem{lemma}[theorem]{Lemma}

\newtheorem{proposition}[theorem]{Proposition}
\newtheorem{remark}[theorem]{Remark}

\numberwithin{equation}{section}

\newcommand{\fr}{\operatorname{Fr}}
\newcommand{\ind}{\mathrm{ind}}

\newcommand{\rr}{\mathbb{R}}
\newcommand{\sph}{\mathbb{S}}

\newcommand{\kk}{\mathbb{K}}
\newcommand{\sect}{\operatorname{Sect}}
\newcommand{\ric}{\operatorname{Ric}}

\begin{document}

\title[Hurewicz fibrations, almost submetries and critical points]{Hurewicz fibrations, almost submetries and critical points of smooth maps}

\author{Sergio Cacciatori}

\address{Universt\`a dell'Insubria -  
Dipartimento di Scienza e Alta Tecnologia\endgraf
Via Valleggio 11, I-22100 Como, Italy
and\smallskip\endgraf INFN, Sezione di Milano, via Celoria 16, I-20133 Milano, Italy.}
\email{sergio.cacciatori@uninsubria.it}

\author{Stefano Pigola}
\address{Universt\`a dell'Insubria -  
Dipartimento di Scienza e Alta Tecnologia\endgraf
Via Valleggio 11, I-22100 Como, Italy}
\email{stefano.pigola@uninsubria.it}
\begin{abstract}
We prove that the existence of a  Hurewicz fibration between certain spaces with the homotopy type of a CW-complex implies some topological restrictions on their universal coverings. This result is used to deduce differentiable and metric properties of maps between compact Riemannian manifolds under curvature restrictions.
\end{abstract}

\date{\today}


\maketitle

\section{Introduction and main results}
The main purpose of the paper is to prove that the existence of a  Hurewicz fibration between certain spaces with the homotopy type of a CW-complex implies certain topological restrictions on their universal coverings. For the sake of brevity, we shall say that the topological space $Z$ is $\kk$-acyclic if its reduced singular homology with coefficients in the field $\kk$ satisfies $\tilde H_n(Z;\kk) = 0$, for every $n \geq 0$.
\smallskip

\begin{theorem} \label{th-general}
Let $X$ be a connected, locally path connected and semi-locally simply connected, separable metric space with finite covering dimension $\dim X <+\infty$ and with the homotopy type of a CW-complex. Let $Y$ be a connected, locally path connected and semi-locally simply connected space with the homotopy type of a finite-dimensional CW-complex. Assume that there exists a Hurewicz fibration $\pi: X \to Y$.
\begin{itemize}
\item [(a)] If at least one fibre $F$ is locally contractible and  $X$ is aspherical then the universal covering space $Y'$ of $Y$ is $\mathbb{K}$-acyclic, for any field $\kk$.
\item [(b)] Let $\widetilde {\Omega_{{y}} Y}$ denote the connected component  of the loop space of $Y$ containing the constant loop $c_{y}\equiv y$. If $X$ is aspherical and the fibre $F=\pi^{-1}(y)$ is a finite dimensional CW-complex, then $H_{k} (\widetilde {\Omega_{y} Y}, \kk) = 0 $ for every  $k > \dim X$ and for any field $\kk$.
\end{itemize}
\end{theorem}
A variant of Theorem \ref{th-general} of special interest can be obtained in case $X$ and $Y$ are genuine CW-complexes. Recall that a finite dimensional CW-complex is a locally contractible (hence locally path connected and semi-locally simply connected), paracompact, normal space of finite covering dimension, \cite{FrPi-book, Miy-Tohoku}.

\begin{theorem}\label{corollary-cw}
Let $X$ and $Y$ be connected, finite dimensional, CW-complexes. Assume that there exists a Hurewicz fibration $\pi: X \to Y$ with at least one locally contractible fibre. If $X$ is aspherical then the universal covering space $Y'$ of $Y$ 
is $\kk$-acyclic, for any field $\kk$.
\end{theorem}

Our main motivation to investigate topological properties of  Hurewicz fibrations is to get information on  maps between compact Riemannian manifolds under curvature restrictions. Indeed, recall that smooth manifolds are finite-dimensional, separable metric space and also finite dimensional CW-complexes. By way of example, we point out the following consequences of Theorem  \ref{th-general} or Theorem \ref{corollary-cw}. First, we consider critical points of smooth maps. Recall that a point $p \in M$ is critical for the $C^{1}$-map $f: M \to N$ if $f$ is not submersive at $p$.
\begin{corollary}\label{corollary-criticalpoints}
Let $f: M \to N$ be a smooth map between compact Riemannian manifolds $(M,g)$ and $(N,h)$ of dimensions $m , n \geq 2$. Assume also that $\sect_M \leq 0$ and $\ric_N \geq K$ for some constant $K>0$. Then $f$ must have a critical point.
\end{corollary}
\begin{proof}
Let $m \geq n$ for, otherwise, the result is trivial. By the Bonnet-Myers theorem, the universal covering space $N'$ of $N$ is compact. Since $N'$ is simply connected, hence orientable, $H_n(N';\rr) \not=0$. On the other hand, by the Cartan-Hadamard theorem $M'$ is diffeomorphic to $\rr^m$, hence $M$ is aspherical. Now, by contradiction, assume that $f:M \to N$ is a smooth map without critical points. Then, the Ehresmann fibration theorem  implies that $f$ is a locally trivial bundle. Since $N$ is (para)compact, $f$ is also a Hurewicz fibration with fibre a smooth $(m-n)$-dimensional manifold. This contradicts Theorem \ref{th-general} or Theorem \ref{corollary-cw}.
\end{proof}

\begin{remark}
\rm{
Existence of critical points for any smooth map $f:\times_1^m \sph^1 \to \mathbb{SO}(3)$ was observed by D. Gottlieb, \cite{Go-Robot}, in relation with critical configurations of multi-linked robot arms.
}
\end{remark}

\begin{remark}\label{rem-curv}
\rm{It is clear from the proof that the role of the curvature of $M$ is just to guarantee that the covering space $M'$ is contractible. Obviously, we do not need that $M'$ is diffeomorphic to $\rr^m$ and therefore the Corollary applies e.g. when $M$ is a compact quotient (if any) of a Whitehead-like manifold. Similarly, the Ricci curvature condition on the target has the purpose to guarantee that the universal covering manifold $N'$ is compact. Thus, we can take e.g. $N$ to be any quotient of the  product $\sph^{n_{1}}\times \sph^{n_{2}}$ with $n_{1},n_{2}\geq 2$. Summarizing, the following differential topological result holds true.}
\end{remark}

\begin{corollary}
 Let $f : M \to N$ be a smooth map between compact differentiable manifolds of dimensions $m , n \geq 2$. If $M$ is aspherical and the universal covering space of $N$ is compact then, necessarily, $f$ has a critical point. 
\end{corollary}

The next application is an estimate of the $\epsilon$-constant for  $e^{\epsilon}$-LcL maps under curvature restrictions.
\begin{definition}
A continuous map $f : X \to Y$ between metric spaces is called an $e^{\epsilon}$-Lipschitz
and co-Lipschitz map ($e^{\epsilon}$-LcL for short), if for any $p \in X$, and any $r > 0$, the metric balls of $X$ and $Y$ satisfy
\[
B^{Y}_{e^{-\epsilon}r}(f(p)) \subseteq f(B^{X}_r(p)) \subseteq B^{Y}_{e^{\epsilon}r}(f(p)).
\]
A $1$-LcL is a submetry in the usual sense of V. Berestovskii.
\end{definition}
Fibration properties of $e^{\epsilon}$-LcL maps has been investigated in \cite{RoXu-advances}. More recent results are contained in \cite{Xu}.

\begin{corollary}\label{corollary-LCL}
Let $M$ and $N$ be compact Riemannian manifolds such that $\sect_M \leq 0$ and $\ric_N \geq K$ for some constant $K >0$. If $f : M \to N$ is an $e^{\epsilon}$-LcL map then $\epsilon > \ln (1.02368)$. 
\end{corollary}
\begin{proof}
By \cite[Theorem A]{Xu} any $(1.02368)$-LcL proper map from a complete Riemannian manifold with curvature bounded from below into any  Riemannian manifold is a Hurewicz fibration. Moreover, all the fibres are locally contractible. On the other hand, as we have already observed above, the curvature assumptions on $M$ and $N$ imply that $M$ is aspherical and $N'$ is not $\rr$-acyclic.
\end{proof}

\begin{remark}
\rm{
The previous result could be stated in the more general setting of Alexandrov spaces. Namely, we can choose $M$ to be a compact Alexandrov-space of finite dimension, with the homotopy type of a CW complex, and  curvature $-C \leq \mathrm{Curv}_M \leq 0$. Indeed, \cite[Theorem A]{Xu} is already stated in this metric setting and, in order to obtain that the universal covering space $M'$ is contractible we can apply the metric version of the Cartan-Hadamard theorem by S. Alexander and R. L. Bishop. See e.g. \cite[Theorem II.4.1 and Lemma II.4.5]{BH-book}.
}
\end{remark}

\begin{remark}
\rm{
Concerning the curvature assumptions what we said in Remark \ref{rem-curv} applies also to Corollary \ref{corollary-LCL}. In fact, we have the following surprisingly general result.
}
\end{remark}

\begin{corollary}
Let $M$ be a compact smooth manifold with contractible universal covering $M'$ and let $N$ be a compact  smooth manifold with compact universal covering $N'$. Then, for any fixed Riemannian metrics $g$ on $M$ and $h$ on $N$ and for every  $e^{\epsilon}$-LcL map  $f:(M,g) \to (N,h)$ it holds $\epsilon > \ln(1.02368)$.
\end{corollary}

The paper is organized as follows: in Section 2 we collect some preliminary facts concerning Hurewicz fibrations, loop spaces and the cohomological dimension of a space in connection with the covering space theory. Section 3 is devoted to the proofs of Theorems \ref{th-general} and \ref{corollary-cw}.

\section{Preliminary results}
In this section we lift a Hurewicz fibration to the universal covering spaces without changing the assumption of the main theorem and, as a consequence, we deduce some information on the singular homology  of the lifted fibre.

\subsection{Lifting Hurewicz fibrations}
In the next result we collect some properties of our interest that a generic covering space inherits from its base space.
\begin{proposition}\label{prop-lift}
Let $P:E' \to E$ be a covering projection between path connected spaces. Then, each of the following properties lifts from the base space $E$ to the covering space $E'$.
\begin{enumerate}
\item[(a)] The space is a CW-complex.
\item[(b)] The space has the homotopy type of a CW-complex.
\item[(c)] The space is regular (and $T_1$) and has finite small inductive dimension.
\item[(d)] The space is locally path connected and $\mathrm{II}$-countable.
\item[(e)] The space is locally path connected, separable and metrizable.
\end{enumerate}
Moreover, if $E'$ is simply connected, i.e. $P:E' \to E$ is the universal covering of $E$, then (b) can be replaced by
\begin{enumerate}
 \item [(b')] The space is Hausdorff, locally path connected and has the homotopy type of a finite dimensional CW-complex.
\end{enumerate}
\end{proposition}

\begin{remark}
\rm{
It will be clear from the proof of (d) that $E'$ is $\mathrm{II}$-countable provided that $E$ has the homotopy type of a $\mathrm{II}$-countable, CW-complex $Y$ with a countable $1$-skeleton $Y^1$. Indeed, the injection $i:Y^1 \to Y$ induces a surjective homomorphism between fundamental groups $i_{\ast} : \pi_1(Y^1, \ast) \to \pi_1(Y,\ast) \simeq \pi_1(E,\ast)$, \cite[Proposition 1.26]{Ha-book} or \cite[Corollary 2.4.7]{FrPi-book}.
}
\end{remark}

\begin{proof}
(a) Indeed, the $n^{\mathrm{th}}$-skeletons $(E')^n$ of $E'$ and $E^n$ of $E$ are related by $P^{-1}(E^n) = (E')^n$ and $P|_{(E')^{n}}:(E')^{n} \to E^{n}$ is a covering projection; \cite[Proposition 2.3.9]{FrPi-book}.\smallskip

(b) Assume that $E$ has the homotopy type of a CW-complex. Since the covering projection $P$ is a Hurewicz fibration, \cite[Theorem II.2.3]{Sp}, and each fibre $P^{-1}(e)$ is a discrete space, hence a CW-complex, the conclusion follows by applying \cite[Theorem 5.4.2]{FrPi-book}. \smallskip

(c) Assume that $E$ is regular, i.e., it is a $T_1$ space with the shrinking property of arbitrarily small open neighborhoods. Then $E'$ is also $T_1$ because for any $e' \in E'$, the fibre $P^{-1}(P(e'))$ is a discrete and closed subspace of $E'$. The shrinking property of $E'$ follows easily from the fact that $P$ is a local homeomorphism. Finally, let us show that the regular space $E'$ has small inductive dimension $\ind(E') \leq n <+\infty$ provided $\ind(E) \leq n <+\infty$. To this end, fix $e' \in E'$  and an open neighborhood $U'$ of $e'$. We have to show that there exists an open set $e' \in V' \subset U$ such that $\ind(\fr_{E'} V') \leq n-1$. Here, $\fr_{X}Y$ denotes the topological boundary of $Y$ as a subspace of $X$. Without loss of generality, we can assume that $P|_{U'} :U' \to U=P(U')$ is a homeomorphism. Since $\ind(U) \leq \ind(E)$, \cite[1.1.2]{En-dim}, by the topological invariance of the inductive dimension we have $\ind(U') \leq n$. By regularity, we can choose an open neighborhood $e' \in U'_0 \subset \overline{U'_0} \subset U'$. It follows from the definition of inductive dimension of $U'$ that there exists an open set (in $U'$ hence in $E'$) $e' \in V' \subset U'_0$ such that  $\ind(\fr_{U'} V') \leq n-1$. To conclude, observe that $\fr_{U'} {V'} = \fr_{E'} {V'}$. 
\smallskip

(d) Assume that $E$ is locally path connected and $\mathrm{II}$-countable. Since $P$ is a local homeomorphism then $E'$ is locally path-connected. Moreover, since the topology of $E$ has a countable basis made up by open sets evenly covered by $P$, to prove that $E'$ is $\mathrm{II}$-countable it is enough to verify that the fibre of the covering projection is a countable set. On the other hand, the fibre $P^{-1}(e)$ is in one-to-one correspondence with the co-set $\pi_1(E,e)/P_{\ast}\pi_1(E',e')$ of the fundamental group $\pi_1(E,e)$ of $(E,e)$. To conclude, we recall that $\pi_1(E,e)$ is  countable  because the topology of the path connected space $E$ has a countable basis $\mathcal{U}$ such that, for any $U_1, U_2 \in \mathcal{U}$, the connected components of $U_1 \cap U_2$ are open (hence countable); see e.g. the proof of \cite[Theorem 8.11]{Lee-top}.
\smallskip

(e) We already know that $E'$ is a ($T_1$-)regular, $\mathrm{II}$-countable (hence separable) topological space. Therefore, we can apply the Urysohn metrization theorem.\smallskip

(b') The Hausdorff property lifts from the base space $E$ to the covering space $E'$: two points on the same fibre are separated by the pre-images of an evenly covered open set of $E$, whereas two points on different fibres are separated by the pre-images of disjoint open sets in $E$. Moreover $E'$ is locally path connected as already observed.\\
Assume that $P:E' \to E$ is the universal covering projection and that the Hausdorff space $E$ has the homotopy type of a CW-complex $Y$ with $\dim Y <+\infty$. Recall that $Y$ is Hausdorff (actually normal) and that the (Lebesgue) covering dimension $\dim Y$ coincides with the dimension of $Y$ as a CW-complex; \cite[Propositions 1.2.1, 1.5.14]{FrPi-book}. Let $Q: Y' \to Y$ be the universal covering of $Y$. We know from (a) that $Y'$ is a CW-complex of dimension $\dim Y' = \dim Y$. We shall show that $E'$ has the same homotopy type of $Y'$, that is:
\begin{lemma}\label{lemma_equivcov}
Let $P_X:X' \to X$ and $P_Y:Y' \to Y$ be universal covering projections between Hausdorff spaces. If $X$ is homotopy equivalent to $Y$ then $X'$ is homotopy equivalent to $Y'$.
\end{lemma}
The proof relies on the following standard fact. Recall that a covering projection $P:X' \to X$ between path connected and locally path connected spaces is said to be normal if the image $P_{\ast}(\pi_1(X',x'))$ of the fundamental group of $X'$ is normal in the fundamental group $\pi_1(X,P(x'))$ of $X$. This condition is trivially satisfied if $X'$ is simply connected.
\begin{lemma}\label{lemma_covtransf}
Let $P:X' \to X$ be a normal covering between  connected, locally path connected and Hausdorff spaces. Let $f:X\to X$ be a continuous map. Let $f',g':X' \to X'$ be continuous liftings of $f$, that is $P \circ f' = f\circ P = P \circ g'$. Then, there exists a covering transformation $\tau \in \mathrm{Deck}(P)$ such that $\tau \circ f' = g'$. 
\end{lemma}
We are now in the position to give the proof of Lemma \ref{lemma_equivcov}. In all that follows, given a homotopy $H:X \times I \to Y$, we shall use the notation $H_{t}:= H(\cdot,t):X \to Y$.
\begin{proof}[Proof of Lemma \ref{lemma_equivcov}]
Let $f:X \to Y$ be a homotopy equivalence. This means that there exist continuous maps $g,h:Y \to X$ such that $g \circ f \simeq \mathbf{1}_X$ via a homotopy $H:X\times I \to X$ and $f \circ h \simeq \mathbf{1}_Y$ via a homotopy $K: Y \times I \to Y$. Choose any liftings
\begin{itemize}
\item $f':X' \to Y'$ of  $(f\circ P_X):X' \to Y$ with respect to the covering projection $P_Y$
\item $g',h':Y' \to X'$ of $(g\circ P_Y), (h\circ P_Y) : Y' \to X$ with respect to the covering projection $P_X$
\item $H':X'\times I \to X'$ of 
$H \circ (P_X \times \mathbf{1}_I):X' \times I \to X$ with respect to $P_X$
\item $K':Y' \times I \to Y'$ of $K \circ (P_Y \times \mathbf{1}_I):Y' \times I \to Y$ with respect to $P_Y$.
\end{itemize}
This is possible, without any restriction on the maps, because each of the domain spaces is simply connected.\\
Now, since
\[
P_X \circ H'_{0} = P_X \circ (g' \circ f')
\]
and since the universal covering is normal, by Lemma \ref{lemma_covtransf} there exists a covering transformation $\tau_1\in \mathrm{Deck}(P_X)$ such that
\[
\tau_1 \circ H'_{0} = g'\circ f'.
\]
Similarly, since
\[
P_X \circ H'_{1} = P_X \circ \mathbf{1}_{X'}
\]
we find $\tau_2 \in \mathrm{Deck}(P_X)$ such that
\[
\tau_2 \circ H'_{1} = \mathbf{1}_{X'}.
\]
Let $H''=\tau_2 \circ H'$ and $g'' = \tau_2\circ \tau_1^{-1}\circ g'$. Then, it follows from the above equations that the homotopy $H''$ realizes the equivalence
\begin{equation}\label{equivcov1}
g'' \circ f' \simeq \mathbf{1}_{X'}.
\end{equation}
Arguing in a similar way, we obtain the existence of $\xi_1,\xi_2 \in \mathrm{Deck}(P_Y)$ such that the homotopy $K''=\xi_2 \circ K'$ realizes the equivalence 
\[
(\xi_2 \circ \xi_1^{-1} \circ f') \circ h' \simeq \mathbf{1}_{Y'}.
\]
Define $h'' = h' \circ (\xi_1 \circ \xi_2 ^{-1})$ and the new homotopy
\[
K''' =( \xi_1 \circ \xi_2^{-1}) \circ K'' \circ (\xi_1 \circ \xi_2 ^{-1} \times \mathbf{1}_I).
\]
Then, it is straightforward to verify that $K'''$ realizes the equivalence
\begin{equation}\label{equivcov2}
f' \circ h'' \simeq \mathbf{1}_{Y'}.
\end{equation}
From \eqref{equivcov1} and \eqref{equivcov2} we conclude that $f'$ is a homotopy equivalence between $X'$ and $Y'$.
\end{proof}
The proof of Proposition \ref{prop-lift} is completed.
\end{proof}

We now verify that a Hurewicz fibration always lifts to a Hurewicz fibration between universal covering spaces. According to the previous Lemma, the lifted fibration enjoys some good properties of the base fibration.
\begin{proposition}\label{prop-lift-hur}
Let $\pi:E \to B$ be a Hurewicz fibration between connected, locally path connected and semi-locally simply connected spaces with the homotopy type of a CW-complex.  Let $P_E: E' \to E$ and $P_B:B' \to B$ denote the universal covering maps of $E$ and $B$ respectively. Then:

\begin{enumerate}
\item [(a)] There exists a Hurewicz fibration $\pi':E' \to B'$ such that  $P_B \circ \pi' = \pi \circ P_E$.
\item [(b)] $E'$ and the fibres $F'$ have the homotopy type of a CW-compex. Moreover, if $E$ is a separable metric space then also $E'$ and $F'$ are separable metric spaces.
\item [(c)] Let  $F_{0} = \pi^{-1}(b_0)$ and $F'_{0}=(\pi')^{-1}(b'_0)$ with $P_B(b'_0) = b_0$. If $E'$ is contractible, then $F'_{0}$ is exactly one of the connected components of $P_E^{-1}(F_{0})$. In particular:
	\begin{enumerate}
	\item [(c.1)] If $F_{0}$ is locally contractible, so is $F'_{0}$.
	\item [(c.2)] If $F_{0}$ is a finite dimensional CW-complex then the same holds for $F'_{0}$.
	\end{enumerate}
\item [(d)] If $E$ is a separable metric space with finite covering dimension then $E'$ and $F'$ have the same properties.
\end{enumerate}
\end{proposition}

\begin{proof}
(a) First, we observe that the composition of Hurewicz fibrations is still a Hurewicz fibration and this, in particular, applies to $\pi'' := \pi \circ P_E :E' \to B$. The verification is straightforward from the definition.
Since $E'$ is simply connected, $\pi''$ lifts to  a continuous map $\pi' : E' \to B'$ such that $$P_B \circ \pi' = \pi''.$$ We show that $\pi'$ is a Hurewicz fibration. To this end, let $H:X \times I \to B'$ be a given homotopy, and assume that $\widetilde {H_{0}}: X \to E'$ is a continuous lifting of $H_{0} : X \to B$ with respect to $\pi'$, i.e.,
\[
\pi' \circ \widetilde{H_{0}} = H_{0}.
\]
Without loss of generality, we can assume that $X$ is connected. If not, we argue on each component. Consider the homotopy $$H''  := P_B \circ H: X \times I \to B'$$ and observe that $\widetilde{H_{0}}$ is a continuous lifting of $H''_{0}:X \to B$ with respect to $\pi''$. Indeed,
\[
\pi'' \circ \widetilde{H_{0}} = P_B \circ (\pi' \circ \widetilde{H_{0}}) = P_B \circ H_{0} = H''_{0}.
\] 
Since $\pi''$ is a Hurewicz fibration, there exists a continuous lifting
\[
\tilde H : X\times I \to E'
\]
 of $H''$ with respect to $\pi''$ such that
 \[
 \tilde H(x,0) = \widetilde{H_{0}}(x)\text{, }\forall x\in X.
 \]
 Note that, in particular,
\[
H'' = \pi'' \circ  \tilde H = P_B \circ (\pi' \circ \tilde H),
\]
namely, $(\pi' \circ \tilde H) : X \times I \to B'$ is  the unique continuous lifting of $H''$ with respect to the covering projection $P_B$ satisfying
\[
(\pi' \circ \tilde H)(x,0) = (\pi' \circ \widetilde{H_{0}})(x) = H_{0}(x).
\]
Since, by definition of $H''$, the map $H:X\times I \to B'$ has the same property we must conclude
\[
\pi' \circ \tilde H = H.
\]
Summarizing, we have proved that $\tilde H$ is a continuous lifting of $H$ with respect to $\pi'$ and satisfies $\tilde H (x,0) = \widetilde{H_{0}}(x)$.\\

(b) By Proposition \ref{prop-lift} (b), both $E'$ and $B'$ have the homotopy type of a CW-complex. Then, by \cite[Proposition 5.4.1]{FrPi-book}, $F'$ has the homotopy type of a CW-complex. Moreover, by (e) of Proposition \ref{prop-lift}, $E'$ is separable and metrizable. The same holds for the subspace $F'$ of $E'$.\\

(c) We shall see in Proposition \ref{prop-contrfibre} below that $F'_{0}$ is path connected. Note that $F'_{0}\subseteq P_E^{-1}(F_{0})$ and $\pi'(P_E^{-1}(F_{0})) \subseteq P_B^{-1}(b_0)$. Since $P_B^{-1}(b_0)$ is a discrete space, the continuous map $\pi'|_{P_{E}^{-1}(F_{0})}$ is constant on the connected components of $P_E^{-1}(F_{0})$. Let $C$  be a component of $P_E^{-1}(F_{0})$ such that $C \cap F'_{0} \not= \emptyset$. Clearly $C \supseteq F'_{0}$ because $F'_{0}$ is connected also as a subspace of $F_{0}$. Since $\pi'(F'_{0})=b'_0$ then, necessarily, we have $\pi'(C) = b'_0$. Thus $C= F'_{0}$.\\
(c.1) Now, assume that $F_{0}$ is locally contractible. Since
\[
P_{F_{0}} = P_E|_{P_E^{-1}(F_{0})}: P_E^{-1}(F_{0}) \to F_{0}
\]
is a covering projection, hence a local homeomorphism, we have that also $P_E^{-1}(F_{0})$ is locally contractible. In particular, this space is locally path connected. It follows that the connected component $C=F'_{0}$ of $P_E^{-1}(F_{0})$ is an open subset and, therefore, $F'_{0}$ inherits the local contractibility of $P_{E}^{-1}(F_{0})$. \\
 (c.2) Finally, suppose that $F_{0}$ is a finite-dimensional CW-complex. Then, using again that $P_{F_{0}}$ is a covering projection, by Proposition \ref{prop-lift} (a) we have that $P_{E}^{-1}(F_{0})$ is a finite dimensional CW-complex. To conclude, recall that $F'_{0}$ is a connected component of $P_{E}^{-1}(F_{0})$, hence a CW-subcomplex.\\

(d) Recall that for a separable, metrizable space the small inductive dimension coincides with the (Lebesgue) covering dimension, \cite{En-dim}. Therefore, we can apply Proposition \ref{prop-lift} (c) and conclude that, if $\dim E <+\infty$ then $\dim E', \dim F'<+\infty$.
\end{proof}

\subsection{The loop space of the base space}
We now consider the path fibration $\varepsilon:\mathcal{P}_{x_{0}}(X)  \rightarrow X$
where
\[
\mathcal{P}_{x_{0}}(X)  =\left\{  \gamma:[0,1]\rightarrow X:\gamma\text{
continuous and }\gamma\left(  0\right)  =x_{0}\right\}
\]
is endowed with the compact-open topology and $\varepsilon$ is the end-point
map:
\[
\varepsilon\left(  \gamma\right)  =\gamma\left(  1\right).
\]
It is a Hurewicz fibration and the fibre $F=\varepsilon^{-1}(x_0)$ is the loop space of $X$ based at $x_{0}$:
\[
\varepsilon^{-1}\left(  x_{0}\right)  =\Omega_{x_{0}}X.
\]
Sometimes, when there is no danger of confusion, the base point is omitted from the notation.

\begin{proposition} \label{prop-contrfibre}
Let $\pi:E \to B$ be a Hurewicz fibrations with fibre $F$ and
path connected spaces $E$ and $B$. Assume that $E$, $B$ (hence $F$) all have the homotopy type of a CW-complex.
\begin{enumerate}
\item[(a)] If $E$ is contractible and $B$ is simply connected, then $F=\pi^{-1}(b)$ is path connected and it is homotopy equivalent to the loop space $\Omega_b B$, for any $b\in B$.
\item[(b)] Assume that $B$ admits the universal covering projection $P_B:B' \to B$. Having fixed $b'_{0}\in B'$, let $P_B(b'_{0})  =b_{0}\in B$ and denote by $\widetilde{\Omega_{b_{0}}
B}$ the connected component of $\Omega_{b_{0}}B$ containing the constant loop $c_{b_{0}}$. Then,
there exists a homotopy equivalence $\theta:\Omega_{b'_{0}}B'\rightarrow\widetilde{\Omega_{b_{0}}B}$.
\end{enumerate}
\end{proposition}

\begin{proof}
(a) The loop space $\Omega_{b}B$ of the simply connected space $B$ is path connected. Indeed, a relative homotopy $H : I \times I \to B$ between the constant loop $c_{b} \in \Omega_{b}B$ and any other loop $\gamma \in \Omega_{b}B$ gives rise to the continuous path $\Gamma: I \to \Omega_{b }B$, $\Gamma(s) = H(\ast,s)$, connecting $c_{b}$ to $\gamma$. Moreover, by \cite[Corollary 3]{Mi-TAMS}, $\Omega_b B$ has the homotopy type of a CW-complex.

Now, since $B$ is path connected then the path space $\mathcal{P}_{b}(B)$ is contractible, \cite[Lemma 3 on p. 75]{Sp}. If follows from  \cite[Proposition 4.66]{Ha-book} that there is a weak
homotopy equivalence $h:F\rightarrow\Omega_{b} B$. In particular, since $h_{\ast}: \pi_{0}(F) \to \pi_{0}(\Omega_{b}B)$ is bijective, then $F$ must be connected. Summarizing, $h$ is a weak homotopy equivalence between connected spaces with the homotopy type of a CW-complex. Therefore, we can apply the classical theorem by J. H. C. Whitehead and conclude that $h$ is a genuine homotopy equivalence, \cite[Theorem 4.5]{Ha-book}.
\smallskip

\noindent (b) Consider the following diagram of fibrations\\

\hspace{4.2cm}
\xymatrix{
B' \ar@{->}[d]_{P_B} \ar@{<-}[r]^{\varepsilon^{\prime}}
&  \mathcal{P}_{b'_{0}}(B^{\prime}) \\ 
B \ar@{<-}[r]_{\varepsilon} &\mathcal{P}_{b_{0}}(B)}\\[3mm]
Since the path space $\mathcal{P}_{b'_{0}}(B')$ is contractible to the
constant loop $c_{b_{0}^{\prime}}$, the continuous map
\[
P_B^{\prime}=P_B\circ\varepsilon^{\prime}:\mathcal{P}_{b'_{0}}(B')\rightarrow B
\]
is homotopic to the constant map $f_{b_{0}} \equiv b_{0}:\mathcal{P}_{b'_{0}}(B')\rightarrow B$. Let $H:\mathcal{P}_{b'_{0}}(B')\times\lbrack0,1]\rightarrow B$ be a homotopy between $f_{b_{0}}$ and
$P_B^{\prime}$ such that $H\left(  \cdot,0\right)  =f_{b_{0}}$ and consider
the constant map $f_{b_{0}}^{\prime}\equiv c_{b_{0}}:\mathcal{P}_{b'_{0}}(B')\rightarrow \mathcal{P}_{b_{0}}(B)  $. Clearly $f_{b_{0}
}^{\prime}$ is a lifting of $f_{b_{0}}$ with respect to $\varepsilon$. Then,
by the homotopy lifting property of the fibration $\varepsilon:\mathcal{P}_{b_{0}}(
B)  \rightarrow B$, there exists a unique lifting $H^{\prime
}:\mathcal{P}_{b'_{0}}(B')\times\lbrack0,1]\rightarrow \mathcal{P}_{b_{0}}(B)$ of $H$
such that $H^{\prime}\left(  \cdot,0\right)  =f_{b_{0}}^{\prime}$. Define
\[
\Theta=H^{\prime}\left(  \cdot,1\right)  :\mathcal{P}_{b'_{0}}(B')\rightarrow \mathcal{P}_{b_{0}}(B)  .
\]
Since, by construction, $\varepsilon\circ H^{\prime}=H$ we have that
$\varepsilon\circ\Theta=P_B^{\prime}$ and the following diagram commutes:

\hspace{4.2cm}
\xymatrix{
B' \ar@{->}[d]_{P_B} \ar@{<-}[r]^{\varepsilon^{\prime}}
&  \mathcal{P}_{b'_{0}}(B^{\prime}) \ar@{->}[d]^\Theta \ar@{->}[ld]^{P'_B}\\
B \ar@{<-}[r]_{\varepsilon} &\mathcal{P}_{b_{0}}(B)}\\[3mm]

Now, define $\theta:= \Theta|_{\Omega_{b'_{0}}B'}:\Omega_{b_{0}^{\prime}}B^{\prime}\rightarrow \mathcal{P}_{b_{0}}(B)$. We claim that, actually,
\[
\theta:\Omega_{b_{0}^{\prime}}B^{\prime}\rightarrow\widetilde{\Omega_{b_{0}}
B} \subseteq \Omega_{b_{0}}B.
\]
Indeed, for any $\gamma^{\prime}\in\Omega_{b_{0}^{\prime}}B^{\prime}$ we
have
\[
\varepsilon\circ\Theta\circ\gamma^{\prime}=P_B^{\prime}\circ\gamma^{\prime
}=P_B\circ\varepsilon^{\prime}\circ\gamma^{\prime}=P_B\left(  b_{0}^{\prime
}\right)  =b_{0},
\]
proving that the end-point of $\Theta\circ\gamma^{\prime}$ is $b_{0}$. On the
other hand, by definition, $\Theta$ takes values in $\mathcal{P}_{b_{0}}(B)
$, thus $\Theta\circ\gamma^{\prime}$ is a path issuing from $b_{0}$.
Summarizing $\Theta\circ\gamma^{\prime}\in\Omega_{b_{0}}B$, i.e.,
$\theta(\Omega_{b_{0}^{\prime}}B^{\prime})\subseteq\Omega_{b_{0}}B$. Actually,
since $B^{\prime}$ is simply connected then $\Omega_{b_{0}^{\prime}}B^{\prime}$
is path connected and $\theta(\varepsilon_{b_{0}^{\prime}})=\varepsilon
_{b_{0}}$. It follows by the continuity of $\theta$ that $\theta(\Omega
_{b_{0}^{\prime}}B^{\prime})\subseteq\widetilde{\Omega_{b_{0}}B}$, as claimed.

We have thus obtained the following commutative diagram

\begin{eqnarray}
\xymatrix{
B' \ar@{->}[d]_{P_B} \ar@{<-}[r]^{\varepsilon^{\prime}}
&  \mathcal{P}_{b'_{0}}(B^{\prime}) \ar@{->}[d]^\Theta &\Omega_{b_{0}^{\prime}}B^{\prime} \ar@{_{(}->}[l]_{\quad\ i'} 
\ar@{->}[d]^\theta \\
B \ar@{<-}[r]_{\varepsilon} &\mathcal{P}_{b_{0}}(B)  \ar@{->}[l] \ar@{<-^{)}}[r]_{\quad\ i}  &  \Omega_{b_{0}}B}\label{diagram_pathcov}
\end{eqnarray}
where $i:\Omega_{b_{0}}B\hookrightarrow \mathcal{P}_{b_{0}}(B)$ and
$i^{\prime}:\Omega_{b_{0}^{\prime}}B^{\prime}\hookrightarrow \mathcal{P}_{b'_{0}}(B^{\prime
})$ are the inclusion maps and each row is a fibration.
Consider the associated homotopy sequences with $j\geq1$ and any $\gamma' \in \Omega'_{b'_{0}}B'$, \cite[Theorem 4.41]{Ha-book}:\\
\xymatrix{
1\ar@{<->}[d]|{\shortparallel} & & & 1\ar@{<->}[d]|{\shortparallel}\\
\pi_{j}(\mathcal{P}B^{\prime},\gamma') \ar@{<-}[r] \ar@{<->}[d]|{\shortparallel} & \pi_{j}(\Omega B',\gamma')\ar[d]^{\theta_*} \ar@{<-}[r]^{\ \Delta'_{j+1}} & \pi_{j+1}(B',b'_{0}) \ar[d]^{(P_B)_*} \ar@{<-}[r]  &  \pi_{j+1}(\mathcal{P}B^{\prime},\gamma') \ar@{<->}[d]|{\shortparallel}\\
\pi_{j}(\mathcal{P}B, \Theta(\gamma')) \ar@{<-}[r]  \ar@{<->}[d]|{\shortparallel}& \pi_{j}(\Omega B, \theta(\gamma')) \ar@{<-}[r]_{\quad \Delta_{j+1}} 
\ar@{<->}[d]|{\shortparallel}&
\pi_{j+1}(B,b_{0})  \ar@{<-}[r] & \pi_{j+1}(\mathcal{P}B',\Theta(\gamma')) \ar@{<->}[d]|{\shortparallel} \\
1 & \pi_{j}(\widetilde{\Omega B}, \theta(\gamma')) & & 1
}\\[0.5cm]

Each row is exact. Moreover, $\pi_{j}\left(  \mathcal{P}B^{\prime}\right)  =1=\pi
_{j}\left(  \mathcal{P}B\right)  $ because these spaces are contractible. It follows
that both $\Delta_{j+1}^{\prime}$ and $\Delta_{j+1}$ are isomorphisms. On the
other hand, from the homotopy exact \ sequence of the covering $P_B:B^{\prime
}\rightarrow B$ we see that $(P_B)_{\ast}:\pi_{j+1}\left(  B^{\prime}
,x_{0}^{\prime}\right)  \rightarrow \pi_{j+1}\left(  B,x_{0}\right)  $ is an
isomorphism for every $j\geq1$. Therefore, we conclude that
\[
\theta_{\ast}=\Delta_{j+1}\circ (P_B)_{\ast}\circ (\Delta'_{j+1})^{-1}: \pi_{j}(\Omega B',\gamma') \to \pi_{j}(\widetilde{\Omega B},\theta(\gamma'))
\]
is an isomorphism for every $j\geq1$ as a composition of isomorphisms.
Finally, since both $\widetilde{\Omega B}$ and $\Omega B^{\prime}$ are path
connected, $\theta_{\ast}$ is in fact an \textquotedblleft
isomorphism\textquotedblright\ also for $j=0$. In conclusion, $\theta$ is a
weak homotopy equivalence of $\widetilde{\Omega B}$ and $\Omega B^{\prime}$.
Since these spaces have the homotopy type of a $CW$-complex, by the theorem of
Whitehead, $\theta$ is a genuine homotopy equivalence.
\end{proof}

\subsection{Singular homology and dimension}
It is well known that the covering dimension dominates the \v{C}ech cohomological dimension of a paracompact space. By the abstract de Rham isomorphism and by duality we therefore deduce the following

\begin{proposition}\label{prop-cech}
Let $X$ be a topological space with the homotopy type of a locally contractible, paracompact space $Y$ of finite covering dimension $\dim Y <+\infty$. Then, for any field $\kk$, the singular homology of $X$ with coefficients in $\kk$ satisfies
\begin{equation}\label{homologyfibre}
H_k(X;\kk) = 0 \text{, } \forall k > \dim Y.
\end{equation}
\end{proposition}

\begin{proof}
Indeed, by the homotopy invariance of the singular homology,
\[
H_{k}(X;\kk) \simeq H_{k}(Y;\kk)\text{  for every }k.
\]
Let $\check{H}^k(Y;\mathcal{K})$ denote the 
\v{C}ech cohomology with coefficients in the constant sheaf $\mathcal{K}$ generated by $\kk$. Since $Y$ is locally contractible and paracompact we have (\cite[Theorem 5.10.1]{Go-topalg}, \cite[Theorem 5.25]{Wa})
\[
\check{H}^k(Y;\mathcal{K}) \simeq \check{H}^k(Y;\kk) \simeq {H}^k(Y;\kk) \simeq H^{k}(Y;\mathcal{K}),
\]
and, moreover, $\check{H}^k(Y;\mathcal{K})=0$ for every $k > \dim Y$, \cite[Section 5.12]{Go-topalg}. Therefore, $H^k(Y;\kk) =0 $ for every $k>\dim Y$. Since we are taking coefficients in a field, by the universal coefficient theorem we have $H^k(Y;\kk) \simeq \mathrm{Hom}(H_k(Y;\kk);\kk)$, \cite[Theorem 53.5]{Mu-elements}, and this implies that $H_{k}(X;\kk) \simeq H_{k}(Y;\kk) = 0$ for every $k > \dim Y$.
\end{proof}

\section{Proof of the main Theorems}
Let us begin with the
\begin{proof}[Proof of Theorem \ref{th-general}]

(a) Let the fibre $F_{0}=\pi^{-1}(y_{0})$ be locally contractible and assume that the total space $X$ of the Hurewicz fibration $\pi:X \to Y$ is aspherical. Since $X$, hence its universal covering space $X'$, has the homotopy type of a CW-complex, this is equivalent to say that $X'$ is contractible. We have to show that the universal covering space $Y'$ of the base space $Y$ is $\kk$-acyclic for every field $\mathbb{K}$. By contradiction,  suppose that this is not the case. Since $Y'$ is simply connected, this means that:
\begin{equation}\label{nonacyclic}
\text{there exists }m\geq 2 \text{ and a field }\kk \text{ such that }H_m(Y';\kk) \not= 0.
\end{equation}
Now, according to Proposition \ref{prop-lift-hur}, let us consider the lifted Hurewicz fibration $\pi':X' \to Y'$. Then, $X'$ and $F'_{0}=(\pi')^{-1}(y'_{0})$, with $P_{Y}(y_{0}')=y_{0}$, are separable metric spaces with the homotopy type of a CW-complex and of finite covering dimension $\dim X' , \dim F'_{0}<+\infty$. Moreover, by (b') of Proposition \ref{prop-lift}, $Y'$ has the homotopy type of a finite dimensional CW-complex and by (c.1) of Proposition \ref{prop-lift-hur}, $F'_{0}$ is locally contractible.
\smallskip

Using \eqref{nonacyclic} joint with Proposition \ref{prop-cech} applied to $Y'$, we deduce that there exists an integer  $n \geq m$ such that $H_n(Y';\kk) \not= 0$ and $H_k(Y';\kk) = 0$, for every $k>n$. Therefore the following result by J.P. Serre, \cite[Proposition 11, p. 484]{Se-annals}  can by applied to the simply connected space $Y'$:
\begin{theorem}\label{th_Serre}
Let $Z$ be a simply connected space and assume that there exists an integer $n \geq 2$ such that the  singular homology of $Z$ with coefficients in a field $\mathbb{K}$ satisfies $H_k(Z;\mathbb{K})=0$ for every $k > n$ and $H_n(Z;\kk) \not=0$. Then, for every $i \geq 0$ there exists $0<j<n$ such that $H_{i+j}(\Omega Z;\kk) \not=0$.
\end{theorem}
In view of this result, the singular homology of the loop space $\Omega_{y'_{0}} Y'$ satisfies
\[
H_{k}(\Omega_{y'_{0}} Y';\kk) \not=0\text{, for infinitely many }k>0.
\]
On the other hand, since we are assuming that $X'$ is contractible, it follows from Proposition \ref{prop-contrfibre} that $F'_{0}$ has the homotopy type of $\Omega_{y'_{0}} Y'$. Whence, using again Proposition \ref{prop-cech} with $X=F'_{0}$, we conclude
\[
H_k(\Omega_{y'_{0}} Y';\kk) \simeq H_k(F'_{0};\kk)  =0 \text{, }\forall k\gg 1.
\]
Contradiction.\\

(b) Assume that $X$ is aspherical (i.e. $X'$ is contractible) and that $F_{0} = \pi^{-1}(y_{0})$ is a finite dimensional CW-complex. Combining Proposition \ref{prop-lift-hur} (c.2) with Proposition \ref{prop-contrfibre} we have that $\widetilde{\Omega_{y_{0}} Y}$ has the same  homotopy type of the finite dimensional CW-complex $F'_{0}=(\pi')^{-1}(y'_{0})$, with $P_{Y}(y'_{0})=y_{0}$. It follows from Proposition \ref{prop-cech} and from the homotopy invariance of the singular homology that
\[
H_k(\widetilde{\Omega_{y_{0}} Y};\kk) \simeq H_k(\Omega_{y'_{0}} Y';\kk) \simeq H_k(F'_{0};\kk)  =0 \text{, }\forall k \geq \dim X.
\]
 The proof is completed.
\end{proof}

As a corollary of the above arguments, we obtain the

\begin{proof}[Proof of Theorem \ref{corollary-cw}]
Assume that $X$ is aspherical. By Proposition \ref{prop-contrfibre} (a) and Proposition \ref{prop-lift-hur} (c.1), $\Omega_{y'_{0}} Y'$ is homotopy equivalent to the locally contractible space $F'_{0}=(\pi')^{-1}(y'_{0})$, with $P_{Y}(y'_{0}) = y_{0}$. Moreover, $F'_{0}$ is paracompact because it is a closed subspace of the CW-complex $X'$, which is paracompact; \cite{Miy-Tohoku}, \cite[Proposition 1.3.5]{FrPi-book}. Finally, since the covering dimension of the CW-complex $X'$ is exactly its CW-dimension, \cite[Proposition 1.5.14]{FrPi-book}, and since $F'_{0}$ is a closed subspace of $X'$ we have $\dim F'_{0} \leq \dim X' < +\infty$, \cite{En-dim}. To conclude, we now use Proposition \ref{prop-cech} as in the proof of Theorem \ref{th-general}.
\end{proof}

\end{document}